\documentclass[10pt]{amsart}
\usepackage{amssymb, amscd, amsmath, amsthm, epsf, epsfig, latexsym, enumerate}

\newtheorem{theorem}{Theorem}
\newtheorem{lemma}[theorem]{Lemma}

\begin{document}

\title{deficiency and commensurators}
% and 4-dimensional infrasolvmanifolds}
 
\author{J.A.Hillman}
\address{School of Mathematics and Statistics\\
     University of Sydney, NSW 2006\\
      Australia }

\email{jonathan.hillman@sydney.edu.au}

\begin{abstract}
We show that if $\pi$ is the fundamental group of a 4-dimensional
infrasolvmanifold then $-2\leq{def(\pi)}\leq0$, 
and give examples realizing each value allowed by our constraints,
for each possible value of the rank of $\pi/\pi'$.
We also determine the abstract commensurators of such groups.
Finally we show that if $G$ is a finitely generated group the kernel 
of the natural homomorphism from $G$ to its abstract commensurator 
$Comm(G)$ is locally nilpotent by locally finite,
and is finite if $\mathrm{def}(G)>1$. 
\end{abstract}



\maketitle

In this note we consider the group-theoretic notion of deficiency  
in two rather different contexts: estimating the deficiencies of the
fundamental groups of 4-dimensional infrasolvmanifolds, 
and in connection with the abstract commensurator of a group.

It is well known that most fundamental groups of closed 3-manifolds
(those without free factors) have balanced presentations,
and thus have deficiency 0.
Every finitely presentable group is the fundamental group
of a closed orientable 4-manifold, and so no such general result 
can be expected in dimension 4 (or higher).
The class of 4-dimensional infrasolvmanifolds is of interest in a variety of contexts, and such manifolds $M$ may be characterized topologically 
by the conditions $\chi(M)=0$ and $\pi=\pi_1(M)$ is virtually poly-$Z$ 
of Hirsch length 4.
(See Theorem 8.1 of \cite{Hi}.)
Since $M$ is aspherical, 
$\pi$ is in fact a $PD_4$-group with $\chi(\pi)=0$,
and is torsion-free.

The first section gives the notation used later.
In \S2 there are a number of simple lemmas which enable us to estimate deficiencies.
Our main results on 4-dimensional infrasolvmanifold groups are in \S3.
We show that $-2\leq{def(\pi)}\leq0$, determine $def(\pi)$ 
when $\beta=\beta_1(\pi;\mathbb{Q})\geq3$, 
and give examples realizing each 
of the three possibilities in the other cases.
(When $\beta=1$ or 2 finding the precise value may involve 
determining conjugacy in $GL(k,\mathbb{Z})$ for $k=2$ or 3, 
which leads to issues of ideal classes in orders 
of quadratic or cubic number fields.
We will not pursue this.)
We also give some simple estimates for poly-$Z$ groups 
of Hirsch length $n>4$.
In \S4 we determine the abstract commensurators of these groups
in terms of the automorphism group of the rational Mal'cev completion 
of the Hirsch-Plotkin radical $\sqrt\pi$
(here the maximal nilpotent normal subgroup).

In the final section we show that if $G$ is finitely generated 
then the kernel $K_G$ of the natural homomorphism from $G$ 
to its abstract commensurator is locally nilpotent by locally finite,
and that if $G$ is finitely presentable and has deficiency $\geq1$ 
then either $K_G$ is finite or $G$ is commensurable 
with $\mathbb{Z}\times{F(2)}$ or $\mathbb{Z}^2$.

I would like to thank R. Inanc Baykur for raising the question which prompted 
the first part of this work.

\section{notation and terminology}

Let $R$ be a ring and $R\Lambda=R[t,t^{-1}]$.
(When $R=\mathbb{Z}$ we write $\Lambda$ instead of $\mathbb{Z}\Lambda$.)
Let $U(n,R)$ be the subgroup of $SL(n,R)$ consisting of upper triangular matrices with 
all diagonal entries 1. (This notation is from \cite{Ro}.)

If $S$ is a subset of a group $G$ let $\langle\langle{S}\rangle\rangle_G$ 
be the normal closure of $S$ in $G$.
Let $G'$, $\zeta{G}$ and $\sqrt{G}$ be the commutator subgroup, 
centre and Hirsch-Plotkin radical (maximal locally-nilpotent, 
normal subgroup) of $G$, respectively.
Let $F(r)$ be the free group of rank $r$.
We shall say that $G$ is $s$-generated if there is an epimorphism 
from $F(s)$ onto $G$,
i.e., if it can be generated by $s$ elements.

A $PD_n$-group $G$ is a Poincar\'e duality group of formal dimension $n$.
Thus the augmentation $\mathbb{Z}[G]$-module has a finite projective resolution of length $n$, $H^i(G;\mathbb{Z}[G])=0$ for $i<n$ and
$H^n(G;\mathbb{Z}[G])$ is infinite cyclic.
(See Chapter III of \cite{Bi}.)
Such a group is {\it orientable over\/} a ring $R$ if $H_n(G;R)\cong{R}$.
If $G$ is a $PD_n$-group with orientation character 
$w:G\to\mathbb{Z}^\times$ and $M$ is a right $R[G]$-module then 
$\overline{M}$ is the left module with structure $g.m=w(g)mg^{-1}$,
for $g\in\pi$ and $m\in{M}$.
If $R=\mathbb{Z}$ or is a field then $\beta_i(G;R)$ is the rank of 
the $R$-module $H_i(G;R)$.

Let $G_1=\mathbb{Z}^3,G_2,G_3$, $G_4$, $G_5$ and $G_6$ and 
$B_1=\mathbb{Z}\times\pi_1(Kb),B_2,B_3$ and $B_4$
be the six orientable and four non-orientable flat 3-manifold groups,
respectively. (These are the torsion-free groups which are virtually 
$\mathbb{Z}^3$.)

Let $\Gamma_q$ be the nilpotent group of class 2 with presentation
\[
\langle{x,y,z}\mid{xz=zx,~yz=zy,~[x,y]=z^q}\rangle,
\]
for some $q\geq1$. 
We shall write $\Gamma$ instead of $\Gamma_1$.

Two groups are {\it commensurable\/} if they have subgroups of finite index 
which are isomorphic.

\section{some elementary lemmas}

The first lemma is straightforward. 

\begin{lemma}
Let $\pi$ be a $PD_4$-group which is orientable over  $R$,
where $R=\mathbb{Z}$ or is a field, and such that $\chi(\pi)=0$. 
Then 
\[def(\pi)\leq2-\beta_1(\pi;R).
\]
\end{lemma}

\begin{proof}
This follows from the standard estimate $def(\pi)\leq\beta_1(\pi;R)-\beta_2(\pi;R)$,
together with the equation $\chi(\pi)=\Sigma_{i=0}^4(-1)^i\beta_i(\pi;R)$ and Poincar\'e duality.
\end{proof}

The next lemma is an immediate consequence of Theorem 2.5 of \cite{Hi}.

\begin{lemma}
Let $\pi$ be a group with first $L^2$ Betti number $\beta_1^{(2)}(\pi)=0$ and $c.d.\pi>2$.
Then $def(\pi)\leq0$.
\qed
\end{lemma}

If $\pi$ is a finitely presentable group with $def(\pi)>0$ 
then every subgroup of finite index also has positive deficiency, 
by the multiplicativity of the Euler characteristic in finite covers.
The corresponding result when $def(\pi)\leq0$ is somewhat weaker.

\begin{lemma}
If $G$ is a finitely presentable group with $def(G)\leq0$ and $H$ is a subgroup 
of finite index in $G$ then 
\[
def(H)\geq[G:H]def(G)+1-[G:H].
\]
\end{lemma}

\begin{proof}
Let $X$ be the 2-complex associated to a presentation of $G$ with optimal deficiency,
and let $X_H$ be the covering space corresponding to $H$.
Then $\chi(X)=1-def(G)$ and $\chi(X_H)=[G:H]\chi(X)$,
so $def(H)\geq1-\chi(X_H)=[G:H]def(G)+1-[G:H]$.
\end{proof}

\begin{lemma}
Let $\pi=K\rtimes\mathbb{Z}$, where $K$ is a finitely presentable group.
If $K$ is $s$-generated then $\pi$ is $(s+1)$-generated and $def(\pi)\geq1+def(K)-s$.
\end{lemma}

\begin{proof}
We may obtain a presentation for $\pi$ from one for $K$ 
by adjoining one more generator $t$ and one more relation $txt^{-1}=w_x$ 
for each $x$ in a set of generators of $K$.
Hence $\pi$ can be generated by $s+1$ elements and $def(\pi)\geq{def(K)+1-s}$.
\end{proof}

Is it generally true that if $K$ is an $s$-generated, finitely presentable group
then the direct product $K\times\mathbb{Z}$ has deficiency $def(K)+1-s$?

If $\pi=K\rtimes\mathbb{Z}$ and $R$ is a ring then
$\pi/K\cong\mathbb{Z}$ acts on $H_*(K;R)$ through conjugation in $\pi$.
The next lemma sharpens part (2) of Lemma 16.11 of \cite{Hi}.

\begin{lemma}
Let $\pi=K\rtimes\mathbb{Z}$, where $K$ is a $PD_3$-group, and let $R=\mathbb{Z}$ or a field.
If $def(\pi)=0$ and $H_1(K;R)$ is $R$-torsion free then $H_1(K;R)$ is cyclic as a $R\Lambda$-module.
\end{lemma}

\begin{proof} 
If $X$ is the finite 2-complex determined by a deficiency 0 presentation for 
$\pi$ then $H_0(X;R\Lambda)=R\Lambda/(t-1)$ and $H_1(X;R\Lambda)$ are $R\Lambda$-torsion modules, 
and $H_2(X;R\Lambda)$
is a submodule of a finitely generated free $R\Lambda$-module.
Hence $H_2(X;R\Lambda)$ is torsion free of rank 1, and so
$H_2(X;R\Lambda)\cong{R\Lambda}$, since $R\Lambda$ is an UFD.
Therefore $H_2(\pi;R\Lambda)$ is cyclic as an $R\Lambda$-module, since it is a quotient of 
$H_2(X;R\Lambda)$,
by Hopf's Theorem.

If $H_1(K;R)$ is $R$-torsion free then it has a short free resolution as a $R\Lambda$-module,
and so $Ext^2_{R\Lambda}(H_1(K;R),R\Lambda)=0$.
Since $\pi$ is a $PD_4$-group, 
\[
H_1(K;R)=H_1(\pi;R\Lambda)\cong
\overline{H^3(\pi;R\Lambda)}
\cong\overline{Ext^1_{R\Lambda}(H_2(\pi;R\Lambda),R\Lambda)},
\]
by Poincar\'e duality and the Universal Coefficient spectral sequence. 
Since $H_2(\pi;R\Lambda)$ is cyclic as a $R\Lambda$-module, so is $H_1(K;R)$.
\end{proof}

In particular, if $K/K'$ is torsion free and $A$ is the matrix of the action of a generator $t$ 
of $\pi/K$ on $K/K'$ with respect to some basis 
then the ``elementary" ideal $E_1(tI-A)$ in $\Lambda$
generated by the maximal proper minors of $\det(tI-A)$ must be the full ring.
(There are subtler conditions that may obstruct cyclicity when this condition holds.
See Chapter 7 of \cite{Hi1}.)

\begin{lemma}
Let $K$ be torsion-free and virtually poly-$Z$ of Hirsch length $3$.
Then $K$ is $3$-generated and $def(K)=0$.
\end{lemma}

\begin{proof}
Such a group  is the fundamental group of a 3-manifold $M$ with one of the geometries
$\mathbb{E}^3$, $\mathbb{N}il^3$ or $\mathbb{S}ol^3$.
Inspection of the ten flat 3-manifold groups shows that six are 2-generated,
while $\mathbb{Z}^3$, $G_2$, $B_1$ and $B_3$ each require three generators.
If $M$ is a $\mathbb{N}il^3$-manifold then it is Seifert fibred,
with base $B$ one of the seven flat 2-orbifolds with no reflector curves
(namely, $T$, $Kb$, $P(2,2)$, $S(2,2,2,2)$, $S(2,4,4)$, $S(2,3,6)$ and $S(3,3,3)$.
The only case that requires any thought is when $B=S(2,2,2,2)$, as otherwise
$\pi^{orb}(B)$ is 2-generated.
In the exceptional case $K$ has a presentation
\[
\langle{a,b,c,d,h}\mid{a^2=h^\alpha,~b^2=h^\beta,~c^2=h^\gamma,~d^2=h^\delta,~abcd=h^e,~h~central}\rangle,
\]
where $\alpha,\dots,\delta$ are odd, and $e+\frac12(\alpha+\beta+\gamma+\delta)\not=0$.
On replacing $a$ by $ah^i$, where $2i=1-\alpha$, we may write $h=a^2$ and then $d=abch^m=abca^{2m}$ 
for $m=\delta-e$, and so $K$ is generated by the images of $\{a,b,c\}$.

If $M$ is a $\mathbb{S}ol^3$-manifold then either it is a mapping torus and
$K\cong\mathbb{Z}^2\rtimes\mathbb{Z}$
or it is the union of two twisted $I$-bundles over  $Kb$ and
 $K\cong\pi_1(Kb)*_{\mathbb{Z}^2}\pi_1(Kb)$.
 In each case $K$ is 3-generated.
 Finally, every 3-manifold has a handle decomposition with equal numbers of 1- and
 2-handles, and so $K$ has a presentation of deficiency 0.
 This is best possible by Lemma 2, since $\beta_1^{(2)}(K)=0$ and $c.d.K=3$.
\end{proof}

An argument similar to one used in the lemma shows that the groups of $\mathbb{N}il^3$-manifolds 
with Seifert base $P(2,2)$, $S(2,4,4)$, $S(2,3,6)$ or $S(3,3,3)$ are 2-generated. 
Note also that Seifert fibred 3-manifolds with base $S(2,2,2,2)$ are also unions 
of two twisted $I$-bundles over $Kb$.
If the base is $T$ or $Kb$ then the group is 3-generated, 
with one exception in each case, when the Euler invariant of the fibration is $\pm1$.

\section{deficiencies of 4-dimensional infrasolvmanifold groups}

Let $\pi$ be a torsion-free virtually poly-$Z$ group of Hirsch length 4.
Then $\beta=\beta_1(\pi;\mathbb{Q})\leq4$.
If $\beta>0$ then $\pi\cong{K}\rtimes\mathbb{Z}$, 
where $K$ is torsion-free and virtually poly-$Z$ of Hirsch length 3.
(If $\beta>1$ the semidirect product structure is not unique.)
Hence $\pi$ is 4-generated and $def(K\rtimes\mathbb{Z})\geq-2$,
by Lemmas 4 and 6.
In particular, this is so if $\pi$ is orientable, since $\chi(\pi)=0$ 
and so $\beta>0$, by Lemma 3.14 of \cite{Hi}.

Since $c.d.\pi=4$ and the $L^2$-Betti numbers $\beta_i^{(2)}(\pi)$ are 0 \cite{CG}, $def(\pi)\leq0$, by Lemma 2.

\subsection{} $\beta=4$. 
There is only one such group, namely $\pi\cong\mathbb{Z}^4$, 
which has a presentation with 4 generators and 6 relations.
Hence $def(\mathbb{Z}^4)=-2$, by Lemma 1.

\subsection{} $\beta=3$.
In this case $def(\pi)\leq-1$, by Lemma 1, and
$\pi$ has a normal subgroup $C\cong\mathbb{Z}$ such that 
$\pi/C\cong\mathbb{Z}^3$.

If $C$ is central then $\pi\cong\Gamma_q\times\mathbb{Z}$ for some $q\geq1$.
If $q>1$ then $def(\pi)=-2$, by Lemmas 1 and 4,
but if $q=1$ then $\Gamma\times\mathbb{Z}$ has the presentation
\[
\langle{t,u,v}|~tu=ut,~tv=vt,~u[u,v]=[u,v]u,~v[u,v]=[u,v]v\rangle,
\]
since the first two relations imply that $t[u,v]=[u,v]t$.
Hence $\Gamma\times\mathbb{Z}$ has deficiency $-1$.
All torsion free nilpotent groups of Hirsch length 4 and class 2 
are commensurable,
and $\Gamma\times\mathbb{Z}$ is the simplest representative.

If $C$ is not central then $\pi$ has a presentation
\[
\langle{t,u,v,x}|~txt^{-1}=x^{-1},~u,v\leftrightharpoons{x},~[t,u]=x^a,~[t,v]=x^b,~[u,v]=x^q\rangle.
\]
Since $tux^it^{-1}=ux^i.x^{a-2i}$ and  $tvx^jt^{-1}=vx^j.x^{a-2j}$,
we may assume that $0\leq{a,b}\leq1$.
Since $tuvt^{-1}=uvx^{a+b}$ we may further assume that $b=0$
(after swapping $u$ and $v$ and replacing $v$ by $uv$, if necessary).
Since $t[u,v]t^{-1}=[u,v]$ we must have $q=0$.

If $a=0$ then $\pi\cong{B_1}\times\mathbb{Z}$ and $\beta_1(\pi;\mathbb{F}_2)=4$.
Hence $def(B_1\times\mathbb{Z})=-2$, by Lemmas 1 and 4.

If $a=1$ then $\pi\cong{B_2}\times\mathbb{Z}$,
which has a presentation
\[
\langle{t,u,v}|~t^2u=ut^2,~tv=vt,~u.tut^{-1}=tut^{-1}.u,~uv=vu\rangle
\]
with deficiency $-1$.
Let $K$ be the subgroup generated by $\{u,tut^{-1},v\}$.
Then $K\cong\mathbb{Z}^3$ and is normal in $\pi$,
and Lemma 4 applies to show that $def(B_2\times\mathbb{Z})=-1$.

\subsection{} $\beta=2$.
In this case $\pi$ has a normal subgroup $N\cong\mathbb{Z}^2$ 
or $\pi_1(Kb)={\mathbb{Z}\rtimes_{-1}\mathbb{Z}}$ and such that 
$\pi/N\cong\mathbb{Z}^2$.
Thus these are the groups of bundles over the torus with fibre the torus or the Klein bottle.
If $N>\pi'$ then $def(\pi)\leq-1$, by Lemma 1.

If $N\cong\mathbb{Z}^2$ then conjugation in $\pi$ induces a  homomorphism 
from $\pi/N$ to $Aut(N)\cong{GL(2,\mathbb{Z})}$.
The image cannot be trivial, for otherwise $\beta>2$.
Hence it is finite cyclic (of order 2, 3, 4 or 6),  $\mathbb{Z}$, $(Z/2Z)^2$ or ${\mathbb{Z}\oplus{Z/2Z}}$.
In each of the latter two cases one summand is generated by $-I$.
(These observations follow easily from the fact that $-I$ generates the centre of ${SL(2,\mathbb{Z})}$,
with quotient $PSL(2,\mathbb{Z})\cong{Z/2Z*Z/3Z}$.)
Hence $\pi\cong{K}\rtimes_\theta\mathbb{Z}$, where $K$ is virtually $\mathbb{Z}^3$,
and has cyclic holonomy.

If  $\pi/N$ acts on $N$ through a cyclic group then $K\cong\mathbb{Z}^3$ 
and so $\pi$ has a presentation
\[
\langle{t,u,x,y}\mid [t,u]=x^m,~txt^{-1}=x^ay^c,~tyt^{-1}=x^by^d,~ux=xu,~uy=yu,
\]
\[
xy=yx\rangle,
\]
where $A=\left(\begin{smallmatrix}
a& b\\
c& d
\end{smallmatrix}\right)$ generates the action of $\pi/N$.
The $\Lambda$-module $N$ is annihilated by $\Delta_A(t)=\det(tI_2-A)$,
while $K/N\cong\Lambda/(t-1)$.
If $A^2\not=I$ this polynomial is irreducible, 
and $N$ is isomorphic to an ideal in the integral domain
$\Lambda/(\Delta_A(t))$.
The extensions of $K/N$ by $N$ (as $\Lambda$-modules) are classified by 
$Ext_\Lambda(\Lambda/(t-1),N)\cong
\Lambda/(t-1,\Delta_A(t))=\mathbb{Z}/(\Delta_A(1))$.
We also have $\pi\cong{L}\rtimes\mathbb{Z}$,
where $L=N\rtimes_A\mathbb{Z}$.

\begin{lemma}
If $\pi\cong\mathbb{Z}^3\rtimes_\theta\mathbb{Z}$,
where $\theta=\left(\begin{smallmatrix}
1&0\\
\mu&A
\end{smallmatrix}\right)\in{GL(3,\mathbb{Z})}$ and
$\mu=\left(\begin{smallmatrix}
m\\
0
\end{smallmatrix}\right)$
then $def(\pi)=0$ if and only if
$A$ is conjugate to $\left(\begin{smallmatrix}
0&\varepsilon\\
1&k
\end{smallmatrix}\right)$ for some $\varepsilon=\pm1$
and $k\in\mathbb{Z}$ and $\mu$ represents a generator of 
$\mathrm{Cok}(I_2-A)$.

If $(mc,m(d-1),(a-1)((d-1)-bc )>1$ then $def(\pi)=-1$.
\end{lemma}

\begin{proof}
If  $def(\pi)=0$ then $K$ is cyclic as a $\Lambda$-module, 
by Lemma 5, and then so is $N=(t-1)K$.
Hence $A$ is conjugate to such a matrix, and 
$Ext_\Lambda(\Lambda/(t-1),N)\cong\mathrm{Cok}(I_2-A)$ as an abelian group.
Conversely, if $A$ is conjugate to $\left(\begin{smallmatrix}
0&\varepsilon\\
1&k
\end{smallmatrix}\right)$ for some $\varepsilon=\pm1$
then $N$ is cyclic, and if also $\mu$ represents a generator of 
$\mathrm{Cok}(I_2-A)$ then $K$ are cyclic.
It follows easily that $def(\pi)=0$.

If the $2\times2$ minors $mc$, $m(d-1)$, and $(a-1)(d-1)-bc)$ of $I_3-\theta$ 
have a common prime divisor  $p$ then 
$\beta_1(\pi;\mathbb{F}_p)\geq3$, and so $def(\pi)\leq-1$, by Lemma 1.
\end{proof}

The $2\times2$ minors are not all 0, since $\beta=2$.

If $A$ has order $n=3$, 4 or 6 it is conjugate 
to such a matrix, and $L\cong{G_3,G_4}$ or $G_5$.
These groups are 2-generated, and so $def(\pi)\geq-1$, by Lemma 4.
If $n=3 $ or 4 and if $m$ is not divisible by 3 or 2, respectively, 
or if $n=6$ then $def(\pi)=0$, by Lemma 5.
In particular, $def(G_5\rtimes\mathbb{Z})=0$.
Otherwise $def(\pi)=-1$.

Similarly, if $A$ is conjugate to $\left(\begin{smallmatrix}
0&1\\
1&0
\end{smallmatrix}\right)$ then $L\cong{B_2}$,
which is 2-generated, and so $def(\pi)\geq-1$,
and $def(\pi)=0$ if and only if $m=\pm1$.

If $A$ is conjugate to $-I$ or
$\left(\begin{smallmatrix}
1&0\\
0&-1
\end{smallmatrix}\right)$ then $\beta_1(\pi;\mathbb{F}_2)\geq3$
and so $def(\pi)\leq-1$.
Since $A\equiv{I_3}$ {\it mod} (2), 
$\beta_1(\pi;\mathbb{F}_2)=4$ if and only if $m$ is even.
In that case $def(\pi)=-2$.

When  $A=\left(\begin{smallmatrix}
1&1\\
0&1
\end{smallmatrix}\right)$ we obtain the group 
$\Upsilon$ with presentation
\[
\langle{t,u,x,y,}\mid{[t,u]=y,~tx=xt,~tyt^{-1}=xy,~uy=yu}\rangle.
\]
Note that conjugating the final relation by $t$ leads to $x.{uy}=uy.x$;
conjugating this by $t^{-1}$ gives $ux=xu$, and so $u,x$ and $y$ all commute.
This is the simplest member of the commensurability class 
of nilpotent groups of Hirsch length 4 and class 3.
(See Theorem 1.5 of \cite{Hi}.) 
Solving for $y=[t,u]$ and then $x=[t,[t,u]]$, 
we obtain the shorter presentation
\[
\langle{t,u}\mid{[t,[t,[t,u]]]=[u,[t,u]]=1}\rangle.
\]
These presentations have optimal deficiency, by Lemma 1, and so $def(\Upsilon)=0$.

Since $\beta_1(G_4\times\mathbb{Z};\mathbb{F}_2)=3$ and 
$G_4\times\mathbb{Z}$ has the presentation 
\[
\langle{t,u,x}|~tu=ut,~tx=xt,~u^2x=xu^2,~xuxu^{-1}=uxu^{-1}x\rangle,
\]
$def(G_4\times\mathbb{Z})=-1$, by Lemma 1.

A more delicate example is given by the group $\pi$ with presentation
\[
\langle{t,x,y,z}\mid {txt^{-1}=xy,}~tyt^{-1}=y^5z^3,~tzt^{-1}=y^{18}z^{11},~x\leftrightharpoons{y,z}\rangle.
\]
Conjugating the relations $x\leftrightharpoons{y,z}$ with $t$ shows that 
$xy\leftrightharpoons{y^5z^3,y^{18}z^{11}}$, and hence that $yz=zy$.
Hence $\langle{x,y,z}\rangle\cong\mathbb{Z}^3$, and so
 $\pi\cong\mathbb{Z}^3\rtimes_\theta\mathbb{Z}$, where
\[
\theta=
\left(
\begin{matrix}
1& 0 & 0\\
1 & 5 & 18\\
 0& 3 & 11
\end{matrix}
\right)\in{SL(3,\mathbb{Z})}.
\]
In this case $\pi/\pi'\cong\mathbb{Z}^2$.
However, 
$\mathbb{Z}^3$ is not cyclic as a $\Lambda$-module,
since $E_1(tI_3-\theta)=(t+1,3)$, 
and so  $def(\pi)=-1$, by Lemma 5.
 
Since $\beta_1(G_2\times\mathbb{Z};\mathbb{F}_2)=4$,  $def(G_2\times\mathbb{Z})=-2$, by Lemma 1.

If the action does not factor through $\mathbb{Z}$ then 
the image has a summand generated by $-I$,
and so we may assume that $K=G_2$.
Hence $\pi$ has a presentation
\[
\langle{t,u,x,y}\mid [t,u]=x^m,~txt^{-1}=x^ay^c,~tyt^{-1}=x^by^d,~uxu^{-1}=x^{-1},
\]
\[uyu^{-1}=y^{-1},~xy=yx\rangle.
\]
If $A=\left(\begin{smallmatrix}
a& b\\
c& d
\end{smallmatrix}\right)$ 
has order 2 we may arrange that $A$ is either  
$\left(\begin{smallmatrix}
1&0\\
0&-1
\end{smallmatrix}\right)$ 
or $\left(\begin{smallmatrix}
0&1\\
1&0
\end{smallmatrix}\right)$, after changing the basis of $N$, if necessary.
The relation $ [t,u]=x^m$ then becomes $[t,u]=x^ry^s$, for some $r,s$ with $(r,s)=m$.
If $A=\left(\begin{smallmatrix}
1&0\\
0&-1
\end{smallmatrix}\right)$ then $\beta_1(\pi;\mathbb{F}_2)\geq3$,
so $def(\pi)\leq-1$,
and $\beta_1(\pi;\mathbb{F}_2)=4$ if and only if $m$ is even.
In particular, $\pi\cong\pi_1(Kb)\times\pi_1(Kb)$ when $m=0$,
and then $def(\pi)=-2$.

If $A=\left(\begin{smallmatrix}
0&1\\
1&0
\end{smallmatrix}\right)$ then $L=N\rtimes_A\mathbb{Z}\cong{B_2}$, 
and so $def(\pi)\geq-1$, by Lemma 4.

Taking $m=0$ and $A=\left(\begin{smallmatrix}
1&2\\
2&5
\end{smallmatrix}\right)$ gives an orientable example with $\beta(\pi;\mathbb{F}_2)=4$, and hence with $def(\pi)=-2$.

The options are more limited when $N=\pi_1(Kb)$,
but we may now determine the deficiency completely.
In this case $\pi$ has a presentation
\[
\langle{t,u,x,y}\mid{[t,u]=x^my^n,
~txt^{-1}=x^\varepsilon{y^k},~uxu^{-1}=xy^\ell,~tyt^{-1}=y^\delta,}
\]
\[
uyu^{-1}=y^\eta,~xyx^{-1}=y^{-1}\rangle,
\]
where $\varepsilon,\eta$ and $\delta$ are each $\pm1$.
Since $[t,u]$ must commute with $y$, the exponent $m$ is even, and so $\beta_1(\pi;\mathbb{F}_2)\geq3$.
Hence $def(\pi)\leq-1$.
On replacing $t$ and $u$ by $tg$, $uh$ for suitable $g,h\in\langle{x,y}\rangle$ we may assume that
$\delta=\eta=1$ and that $0\leq{k,\ell}\leq1$.

If $k$ or $\ell$ is 1 we may solve for $y$, and get a presentation of deficiency $-1$.
For then the subgroup generated by $\{u,x\}$ is normal in $\pi$,
and so $ty=yt$ is a consequence of the first two relations.
Hence $def(\pi)=-1$.

If $k=\ell=0$ then $x$ also commutes with $[t,u]$, so $n=0$.
Hence $\beta_1(\pi;\mathbb{F}_2)=4$, and so $def(\pi)=-2$.

\subsection{} $\beta=1$. 
In this case there is an unique normal subgroup $K$ such that $\pi\cong{K}\rtimes\mathbb{Z}$.
If $K=\pi'$ then $\pi$ is a 2-knot group,
and if $K=\pi'\cong\mathbb{Z}^3$ then $\pi$ has deficiency $\geq-1$ \cite{AR}. 
It has deficiency 0 if and only if $\pi'$ is cyclic as a $\Lambda$-module,
by Theorem 16.12 of \cite{Hi}. 
The presentations 
\[
\langle a,b\mid ba^2=a^3b^2,~b^2a=a^2b^3\rangle
\]
and
\[
\langle t,x,y,z\mid xz=zx,\medspace yz=zy,\medspace 
txt^{-1}=y^{-5}z^{-8}\! ,\medspace tyt^{-1}=y^2z^3\! ,\medspace 
tzt^{-1}=xz^{-7}\rangle.
\]
illustrate the possibilities.

In fact all torsion-free solvable 2-knot groups have deficiency $-1$ or 0. 
(See Chapter 16 of \cite{Hi}.) 
Thus examples with $\beta=1$ and deficiency $-2$
must have abelianization $\mathbb{Z}\oplus{T}$, 
where $T$ is a nontrivial finite group.
Let $M$ be a $\mathbb{N}il^3$-manifold with Seifert base $S(2,2,2,2)$.
Then $\pi=\pi_1(M)\times\mathbb{Z}$ is orientable, 
and $\beta_1(\pi;\mathbb{F}_2)=4$.
Hence $def(\pi)=-2$, by Lemma 1.
The group $\mathbb{Z}^3\rtimes_{-I}\mathbb{Z}$ is a simpler example, 
but is non-orientable.
 
\subsection{} $\beta=0$.
Every such group is non-orientable, and is an extension of 
the infinite dihedral group $D_\infty$ by a normal subgroup $K$, 
by Lemma 3.14 of \cite{Hi}.
Hence $\pi\cong{G*_KH}$ where $G,H$ and $K$ are subgroups of Hirsch length 3,
and $[G:K]=[H:K]=2$.
Since $G,H$ and $K$ are torsion free and virtually poly-$Z$ of Hirsch length 3,
they each have deficiency 0,
and can be generated by at most $3$ elements, by Lemma 6.
Hence $\pi$ is 4-generated. 

\begin{lemma}
Let $G$ be a torsion-free virtually poly-$Z$ group of Hirsch length $3$,
and let $K$ be a subgroup of index $2$. 
Then $G$ can be generated by three elements $u,v,w$ 
such that $u^2,v,w$ generate $K$.
\end{lemma}

\begin{proof}
If $K$ is 2-generated there is nothing to prove, 
so we may assume that either $G$ maps onto $\mathbb{Z}$, 
or that $G$ maps onto $D_\infty$.
Thus $G$ has a normal subgroup $N$ such that $G/N\cong\mathbb{Z}$ or $D_\infty$.
Since $N$ has Hirsch length 2 it is either $\mathbb{Z}^2$ or $\pi_1(Kb)$.
If $K\cap{N}$ has index 2 in $N$ the result follows from the corresponding result for groups of Hirsch length 2, i.e., 
that $N$ can be generated by two elements $x,y$ such that  
$x^2,y$ generates $K\cap{N}$.
 
Thus we may assume that $K\cap{N}=N$.
If $G/N\cong\mathbb{Z}$ let $t\in{G}$ represent a generator of $G/N$.
Then $K/N$ is generated by $t^2$ and the result is clear.
 If $G/N\cong{D_\infty}$ let $u,v\in{G}$ represent generators of order 2 for $G/N$.
Then $P=\langle{N,u}\rangle$ and $Q=\langle{N,v}\rangle$ have Hirsch length 2,
and $G\cong{P*_NQ}$. 
We see easily that $N$ can be generated by $u^2$ or $v^2$ and one other element.
Considering the subgroups of index two in $D_\infty$, we see that $K/N$ is generated by
either $u$ and $(uv)^2$ or $v$ and $(uv)^2$. 
in the first two cases the result is clear.
In the third case it is also clear,
since $G/N$ is generated by $u$ and $uv$, and $u^2$ is a generator for $N$.
\end{proof}

Using this lemma we may obtain a presentation for $\pi$ 
by adjoining one generator $u$ for $G$ and three relations $u^2\in{K}$,  
$uvu^{-1}\in{K}$ and $uwu^{-1}\in{K}$ 
to a presentation for $H$ of deficiency 0.
Hence $def(\pi)\geq-2$.

There are four flat 4-manifold groups with $\beta=0$.
Presentations for these groups are given in \S4 of Chapter 8 of \cite{Hi}.
The presentation given for $G_2*_\phi{B_2}$ has deficiency $-1$.
However the final relation is redundant, 
and this group has a presentation
\[
\langle{s,t}\mid{st^2s^{-1}=t^{-2}},~sts^2t^{-1}s^{-1}=ts^{-2}t^{-1}\rangle.
\]
Hence $def(G_2*_\phi{B_2})=0$.
If $\pi=G_2*_\phi{B_1}$, $G_2*_\phi{G_2}$ or $G_6*_\phi{B_3}$
then $ \beta_1(\pi;\mathbb{F}_2)=3$, and so $def(\pi)\leq-1$ in each case, by Lemma 1.
Hence the first two of these groups have deficiency $-1$.
The presentation given for $G_6*_\phi{B_3}$ has deficiency $-2$,
but we do not know whether this is optimal.

Let $M$ be the $\mathbb{N}il^3$-manifold with Seifert base $S(2,2,2,2)$
and generalized Euler invariant $\frac12$.
Then $K=\pi_1(M)$ has $\Gamma$ as a subgroup of index 2.
Let $\pi=K*_\Gamma{K}$ be the amalgamated free product 
of two copies of $K$ over this subgroup.
Then $\beta=0$, but $\beta_1(\pi;\mathbb{F}_2)=4$, and so $def(\pi)=-2$.

\subsection{}
In higher dimensions the situation is much less clear.

Let $\pi$ be a torsion free, virtually poly-$Z$ group of Hirsch length $n$.
Then $\pi$ is $(g+1)$-generated if every finite quotient of $\pi$ is $g$-generated, 
and there are such groups which are not $g$-generated \cite{LW}.
Are all such groups $n$-generated?

Let $\pi^\tau$ be the preimage in $\pi$ of the torsion subgroup of $\pi/\pi'$,
and let $\gamma$ be the rank of $\pi^\tau/[\pi,\pi^\tau]$.
Then $\gamma\leq{n}-\beta$.
It follows from the five-term exact sequence of low degree
for $\pi$ as an extension of $\mathbb{Z}^\beta$ by $\pi^\tau$
that $\beta_2(\pi;\mathbb{Z})\geq\binom\beta2-\gamma$.
Since $def(\pi)\leq\beta-\beta_2(\pi;\mathbb{Z})$, we have
a universal upper bound
\[
def(\pi)\leq{n}-\binom\beta2.
\]
(In particular, if $\beta\geq\frac12+\sqrt{3n}$ then $def(\pi)\leq2-\lfloor\frac{n+1}2\rfloor$.)

An easy induction using Lemma 4 shows that 
if $\pi$ is a poly-$Z$ group then it 
has a presentation with $n$ generators and $\binom{n}2$ relations,
and so 
\[
n-\binom{n}2\leq{def}(\pi).
\]
If $\pi=\mathbb{Z}^n$ then these bounds agree, 
and so $def(\mathbb{Z}^n)=n-\binom{n}2$.

Finitely generated torsion free nilpotent groups 
form an important special case.
These are poly-$Z$ groups.
The relatively free nilpotent groups $F(\beta)/F(\beta)_{[k+1]}$ 
have natural presentations with $\beta$ generators 
and $r_{k+1}(\beta)$ relations, 
where
\[
r_k(\beta)=\frac1k\Sigma_{d|k}\mu(d)\beta^{\frac{k}d}.
\]
is the rank of the section $F(\beta)_{[k]}/F(\beta)_{[k+1]}$, 
by Witt's formulae.
(Here $\mu$ is the M\"obius function. 
See Theorems 5.11 and 5.12 of \cite{MKS}.)
It follows immediately from the homology exact sequence of low degree
for $F(\beta)$ as an extension of $F(\beta)/F(\beta)_{[k+1]}$ by $F(\beta)_{[k+1]}$ that
\[
H_2(F(\beta)_{[k]}/F(\beta)_{[k+1]};\mathbb{Z})\cong
{F(\beta)_{[k+1]}}/F(\beta)_{[k+2]}.
\]
(This is a special case of Hopf's formula.)
Hence these presentations are optimal.
The Hirsch length of $F(\beta)/F(\beta)_{[k+1]}$ is 
\[
n(\beta,k)=\Sigma_{j\leq{k}}r_k(\beta)=
\Sigma_{st\leq{k}}\frac1{st}\mu(s)\beta^t.
\]
We may assume that $\beta>1$.
The leading term is $\frac{\beta^k}k$, and ${|n(\beta,k)|-\frac{\beta^k}k|<\beta^{k-1}}$.
Hence
\[
def(F(\beta)/F(\beta)_{[k+1]})=
\beta-r_{k+1}(\beta)\sim-\frac{k}{k+1}\beta{n(\beta,k)}.
\]
(This estimate is less than $n(\beta,k)-\binom\beta2$, for all $\beta\leq{n(\beta,k)}$ and $k\geq1$.)
Among  all $\beta$-generated nilpotent groups of nilpotency class $k$,
the relatively free group has maximal Hirsch length 
and surely requires the fewest relations,
and so this may be close to a universal upper bound 
for the deficiencies of such groups.

When $k\leq2$ and $\pi/\pi'\cong\mathbb{Z}^\beta$ we can be more precise.
If $\beta\leq{n}\leq\binom{\beta+1}2$ then
the torsion free quotients of $F(\beta)/F(\beta)_{[3]}$ 
by central subgroups $C$ of rank $\binom{\beta+1}2-n$ 
have Hirsch length $n$ and presentations with $\beta$ generators 
and $(\binom{\beta+1}2-n)+\beta(n-\beta)$ relations.
(The first term corresponds to a basis for the kernel $C$,
while the second  reflects the fact that the commutator subgroup is central.)
If $\pi$ is such a quotient of $F(\beta)/F(\beta)_{[3]}$  then 
\[
n-\binom\beta2-\beta(n-\beta)\leq{def(\pi)\leq{n}}-\binom\beta2.
\]
Note that  $k=1\Leftrightarrow\beta=n\Leftrightarrow\pi=\mathbb{Z}^n$,
and $def(\mathbb{Z}^n)=n-\binom{n}2$ minimizes the lower bound.
(Moreover, the lower bound is  less than $-\frac23\beta{n}$ if and only if $3<\beta<\frac23n$.)

If  $\pi=\Gamma\times\mathbb{Z}$,
then $\beta=3$ and $n=4$, and $def(\Gamma\times\mathbb{Z})=-1$, 
which is strictly between these bounds.
The estimate for the number of relations used above 
overlooks the fact that if one generator $t$ commutes with two others 
$u$ and $v$ then it commutes with $[u,v]$.
Thus  the relation $t[u,v]=[u,v]t$ is redundant if $[u,v]\notin{C}$.
When $C$ has a basis $B$ consisting of commutators $[x_i,x_j]$, 
where $\{x_1,\dots, x_\beta\}$ is a basis for $F(\beta)$, the lower bound should be increased
by the difference between the number of triples $i<j<k$ with two of $[x_i,x_j]$, $[x_i,x_k]$,  and
$[x_j,x_k]$ in $B$ and the number of  such triples for which all three are in $B$.
Can this correction be formulated in a manner independent of special bases?
Is it the rank of the subspace of Massey triple products in $H^2(\pi;\mathbb{Z})$?

From another perspective, if $G$ is a nilpotent group
with a minimal generating set of $g\geq3$ elements 
then $def(G)<g-\frac{g^2}4$, by Theorem 2.7 of \cite{Lu}.
(Lubotzsky's argument extends the Golod-Shafarevitch inequality
to the pro-$p$ completion of $G$, for suitable primes $p$.)

An example which is far from nilpotent is given by
the group $G$ with presentation 
\[
\langle{t,x}\mid{xt^ixt^{-i}=t^ixt^{-i}x ~\forall~1\leq{i}<\lfloor\frac{n+1}2\rfloor,~t^{n-1}xt^{1-n}=xw}\rangle,
\]
where $n>1$, and $w$ is a word in the conjugates $t^ixt^{-i}$ with $1\leq{i}\leq{n-2}$. 
Let $k=\lfloor\frac{n-1}2\rfloor$.
Conjugating the relations $xt^ixt^{-i}=t^ixt^{-i}x $ by powers of $t$ 
shows that $t^kxt^{-k}$ commutes with all its conjugates.
It follows easily that $G\cong\mathbb{Z}^{n-1}\rtimes_T\mathbb{Z}$,
where $T$ is the companion matrix for a monic polynomial with constant term $\pm1$.
Thus $G$ is a 2-generated poly-$Z$ group of Hirsch length $n$,
and ${def}(G)\geq2-\lfloor\frac{n+1}2\rfloor$.

If $w=1$ then $G$ has a free abelian subgroup of rank $n$ and index $n-1$,
and so 
\[
(n-1)(def(G)-1)\leq{def(\mathbb{Z}^n)-1}=-\frac12(n-1)(n-2),
\]
by Lemma 3.
Hence $def(G)\leq2-\frac{n}2$.
Since $def(G)$ is an integer, we see that the above presentation is optimal,
and so ${def}(G)=2-\lfloor\frac{n+1}2\rfloor$.

Does every torsion free, virtually poly-$Z$ group $\pi$ of Hirsch length $n$
satisfy
\[
n-\binom{n}2\leq{def(\pi)}\leq2-\lfloor\frac{n+1}2\rfloor?
\]

\section{abstract commensurators of poly-$Z$ groups of Hirsch length 4}

If $G\leq\pi$ the {\it commensurator of $G$ in $\pi$} is
\[Comm_\pi(G)=\{\alpha\in\pi\mid [G:G\cap\alpha G\alpha^{-1}]<\infty,
~[G:G\cap\alpha^{-1} G\alpha]<\infty\}.\]
The {\it abstract commensurator\/} of a group $G$ is $Comm(G)$, the group
of equivalence classes of isomorphisms $\alpha:H\cong J$ between subgroups of
finite index in $G$, where $\alpha$ and $\alpha'$ are equivalent if they agree
on some common subgroup of finite index, and the product of the equivalence
classes of $\alpha$ and $\beta$ is represented by the partially defined
composite $\alpha\circ\beta$.
Conjugation in $G$ determines a natural homomorphism $\gamma:G\to{Comm(G)}$,
and $Comm_\pi(G)=\{\alpha\in\pi\mid{c_\alpha}\in Comm(G)\}$,
where $c_\alpha$ is conjugation by $\alpha$ in $\pi$.
If $H$ has finite index in $G$ then we may identify $Comm(H)$ with $Comm(G)$. 
Therefore if $G_1$ and $G_2$ are commensurable $Comm(G_1)\cong Comm(G_2)$.
If $G$ has only finitely many subgroups of finite index then
on letting $H$ be their intersection we find that 
$Comm_\pi(G)=N_\pi(H)$ and $Comm(G)=Aut(H)$. 

Every finitely generated torsion free nilpotent group is commensurate with
a subgroup of  the upper triangular matrix group $U(n,\mathbb{Z} )$ (for $n$ large enough).
Conversely, every subgroup of $U(n,\mathbb{Z})$ is a finitely generated torsion free nilpotent group.
The corresponding subgroups $N_\mathbb{Q}$ of  $SL(n,\mathbb{Q})$ are the rational Mal'cev completions  \cite{Ma}.
It is well known that $Comm(\mathbb{Z}^n)\cong{Aut(\mathbb{Q}^n)=GL(n,\mathbb{Q})}$;
a simple extension of this result settles the nilpotent case.

\begin{theorem}
Let $N$ be a finitely generated torsion free nilpotent group.
Then $Comm(N)\cong{Aut(N_\mathbb{Q})}$.
\end{theorem}

\begin{proof}
We may assume that $N\leq{U(n,\mathbb{Z})}$ and that
$N/N'$ is torsion free (after passing to a subgroup of finite index, if necessary).
Then $N$ is generated by elements representing a basis for $N/N'$.
If $J$ is a subgroup of finite index in $N$ then the inclusion $i_J$ induces an isomorphism
$i_{J\mathbb{Q}}:J_\mathbb{Q}\cong{N_\mathbb{Q}}$.
If $\alpha:J\to{N}$ is injective then $\alpha(J)$ also has finite index in $N$.
The composite $f(\alpha)=i_{f(J)\mathbb{Q}}\alpha_\mathbb{Q}i_{J\mathbb{Q}}^{-1}$ is an automorphism of $N_\mathbb{Q}$.
Clearly $f$ defines an injective homomorphism from $Comm(N)$ to $Aut(N_\mathbb{Q}$.
Conversely, if $g$ is an automorphism of $N_\mathbb{Q}$ then the images of the generators of $N$
under $g$ have bounded denominators.
Hence there is a subgroup $J$ of finite index such that $g(J)\leq{N}$,
which again has finite index in $N$, and so $f$ is onto.
\end{proof}

The situation is somewhat more complicated when $\pi$ is not virtually nilpotent.
By a fundamental theorem of Mal'cev, 
every virtually poly-$Z$ group has a subgroup of finite index with
nilpotent commutator subgroup.
(See Proposition 15.1.6 of \cite{Ro}.)
In our situation, $\sqrt\pi$ has Hirsch length 3 or 4,
by Theorem 1.6 of \cite{Hi}, and so $\pi/\sqrt\pi$ is virtually $\mathbb{Z}$ or finite.
Thus if $\pi$ is not virtually nilpotent it is commensurate with a semidirect product
$\sqrt\pi\rtimes\mathbb{Z}$, where $\sqrt\pi\cong\mathbb{Z}^3$ or $\Gamma$.
There are three cases: 
 either (i) $\pi\cong\mathbb{Z}^3\rtimes_A\mathbb{Z}$ and no eigenvalue of $A$ is a root of unity;
or (ii) $\pi\cong\mathbb{Z}\times(\mathbb{Z}^2\rtimes_B\mathbb{Z})$
where no eigenvalue of $B$ is a root of unity;
or (iii) $\pi\cong\Gamma\rtimes_\rho\mathbb{Z}$ and the automorphism $\rho^{ab}$
induced on the abelianization has no eigenvalues which are roots of unity.

When $\sqrt\pi\cong\mathbb{Z}^3$ and no eigenvalue of $A$ is a root of unity
it is fairly easy to give representatives for the classes in $Comm(\pi)$.
Let $t\in\pi$ represent a generator of $\pi/\sqrt\pi$ and be such that conjugation by $t$ induces the automorphism $A$.
Then  $\pi$ has a cofinal system of finite index subgroups of the form $\langle{m}\sqrt\pi,t^kw\rangle$,
with $w\in{m}\sqrt\pi$.
(Note that  $(t^kw)^\ell=t^{k\ell}\nu_\ell(A^k)w$, for $\ell\geq1$,
where $\nu_\ell(X)$ is the polynomial $(X^\ell-1)/(X-1)=\Sigma_{i<\ell}X^i$.
Hence $\nu_\ell(A^k)w\in{m}\sqrt\pi$ for some $\ell$, 
since $\sqrt\pi/m\sqrt\pi$ is finite.)
Since $A$ has no eigenvalues which are roots of unity, no two distinct powers of $t$ are conjugate.
(In particular, $A$ is not conjugate to its inverse.)
Therefore a homomorphism $\alpha$ from one of these subgroups into $\pi$ must map $t^kw$ to $t^kw'$ 
for some $w'\in\sqrt\pi$.
The partial isomorphism $\alpha$ extends to an automorphism $\alpha_\mathbb{Q}$
of $\sqrt\pi_\mathbb{Q}\cong\mathbb{Q}^3$ which commutes with $A^k$.
Thus $\alpha$ determines a triple $(B,w,k)$ with $B\in{GL(3,\mathbb{Q})}$ such that $BA^k=A^kB$
and $w\in\mathbb{Z}^3$.
Conversely, if $B\in{GL(3,\mathbb{Q})}$ then the entries have a common denominator, and so 
the restriction of $B$ to $m\mathbb{Z}^3$ has image in $\mathbb{Z}^3$.
Thus every such triple represents an element of $Comm(\pi)$.
We must also consider the equivalence relation necessary to cope with restriction to a subgroup
$\langle{mp}\mathbb{Z}^n,t^{k\ell}\rangle$, where $\ell>1$.
Define an equivalence relation $\sim$ on such triples $(B,w,k)$ 
by $B(w,k)\sim(B,\nu_\ell(A^k),k\ell)$ for all $\ell\geq1$.
Then $Comm(\pi)$ is the set of equivalence classes, with composition represented by
$(B,w,k)(B',w',k)=(BB',w+Bw',k)$, where we may assume $k$ so large that $Bw'\in\sqrt\pi$ also.

The other two cases are similar, except that we must now allow for the possibility
that $B$  or $\rho$ is conjugate to its inverse.
We shall not discuss them further.

\section{abstract commensurators and deficiency}

Let $Hol(G)=G\rtimes{Aut(G)}$, with the natural action of $Aut(G)$ on $G$.
If $G\leq\pi$ let $C_\pi(G)$ be the centralizer of $G$ in $\pi$,
$\gamma_\pi(G):Comm_\pi(G)\to Comm(G)$ be the natural homomorphism 
and $K_\pi(G)=\mathrm{Ker}(\gamma_\pi(G))$.
Let $K_{aut}(G)=Aut(G)\cap{K_{Hol(G)}(G)}$ be the kernel 
of the natural homomorphisms from $Aut(G)$ to $Comm(G)$.
Clearly $K_\pi(H)=H\cap{K_\pi(G)}$ if $H$ is a subgroup of finite index in $G$.

\begin{lemma}
Let $G\leq\pi$ be a pair of groups, with $G$ finitely generated.
Then $K_\pi(G)=\cup{C_\pi(H)}$ and $K_G=\cup{C_G(H)}$,
where the unions are taken over all normal subgroups $H$ of finite index in $G$.
\end{lemma}

\begin{proof}
This is clear, since normal subgroups are cofinal among all
subgroups of finite index in a finitely generated group.
\end{proof}

\begin{theorem}
Let  $G$ be a finitely generated group. Then
\begin{enumerate}

\item ${K_G}'$ and $K_G/\sqrt{K_G}$  are locally finite:

\item if $G$ has no nontrivial locally-finite normal subgroup
then $K_G$ is torsion-free abelian;

\item if $C_G(H)$ is finite for all normal subgroups $H$ of finite index in $G$
then $K_G$ and $K_{aut}(G)$ are locally finite;

\item if $G$ has no infinite elementary amenable normal subgroup 
then $K_G$ is the maximal finite normal subgroup of $G$,
and $Comm(G)$ maps injectively to $Comm(G/K_G)$.
\end{enumerate}
\end{theorem}

\begin{proof}
(1) Let $H$ be a subgroup of finite index in $G$.
Then $[C_G(H):\zeta{C_G(H)}]\leq{[C_G(H):C_G(H)\cap{H}]}\leq[G:H]$,
since $C_G(H)\cap{H}\leq\zeta{C_G(H)}$.
Therefore $C_G(H)$ has centre of finite index,
and so $C_G(H)'$ is finite, by Schur's Theorem
(Proposition 10.1.4 of \cite{Ro}).
Let $H_n$ be a cofinal descending sequence of normal subgroups of finite index 
in $G$, and let $K_n=C_G(H_n)$.
Then $K_G=\cup K_n$, by Lemma 10, and so ${K_G}'=\cup{K_n}'$ is locally finite.
Since $K_n$ is  normal in $K_{n+1}$,  $\sqrt{K_n}\leq\sqrt{K_{n+1}}$, for all $n\geq1$,
and so $\cup\sqrt{K_n}\leq\sqrt{K_G}$.
Since $\zeta{K_n}\leq\sqrt{K_n}$ and $[K_n:\zeta{K_n}]$ is finite, for all $n$,
the quotient $K_G/\sqrt{K_G}$ is also locally finite.
(In particular, $K_G$ is elementary amenable.)

2) This follows from (1), since the preimage in $G$ of the torsion of ${K_G}'$ is a 
locally-finite normal subgroup of $G$.

(3) If $C_G(H)$ is finite for all normal subgroups $H$ of finite index in $G$
then $K_G=\cup{C_G(H)}$ is locally finite.
Let $F$ be a finitely generated subgroup of $K_{aut}(G)$.
Then there is a subgroup $H$ of finite index in $G$ such that 
$\alpha(h)=h$ for all $\alpha\in F$ and $h\in H$.
We may assume without loss of generality that $H$ is normal in $G$.
Let $f_\alpha(g)=g\alpha(g^{-1})$ for all $g\in G$ and $\alpha\in{F}$.
Then $f_\alpha(gh)=f_\alpha(g)$ and 
$hf_\alpha(g)h^{-1}=f_\alpha(hg)=f_\alpha(gg^{-1}hg)=f_\alpha(g)$ 
for all $g\in G$ and $h\in H$.
Therefore $f_\alpha$ factors through the finite group $G/H$,
and takes values in $C_G(H)$.
Since $C_G(H)$ is finite and $f_\alpha=f_\beta$ 
if and only if $\alpha=\beta$ it follows that $F$ is finite.
Thus $K_{aut}(G)$ is locally finite.

(4) If $G$ has no infinite elementary amenable normal subgroup
then $K_G$ is finite. 
Moreover if $F$ is the maximal locally finite normal subgroup of $G$
it is finite, and acts trivially on $C_G(F)$, which has finite index in $G$.
It follows easily that $K_G=F$.

Let $f:H\to{J}$ be an isomorphism of subgroups of finite index in $G$.
Then $f$ restricts to an isomorphism $K_H=H\cap{K_G}\to{K_J=J\cap{K_G}}$, 
and so induces an isomorphism 
$\overline{f}:\overline{H}=H/H\cap{K_G}\cong\overline{J}=J/J\cap{K_G}$.
Hence $\rho([f])=[\overline{f}]$ defines a homomorphism from $Comm(G)$ to $Comm(G/K_G)$.
If $\overline{f}=1$ then we may assume that $J=H$ and that $f$ induces the identity 
on $\overline{H}$.
But it is then easy to see that $f$ restricts to the identity on a subgroup of finite index in $H$,
since $K_H=H\cap{K_G}$ is finite.
Hence  $\rho:Comm(G)\to{Comm(G/K_G)}$ is injective.
\end{proof}

In particular, if $\sqrt{G}$ is locally finite then $K_G$ is locally finite.
If, moreover, $G\leq{GL(n,\mathbb{Q})}$ then $K_G$ must be finite, 
since locally finite subgroups of such linear groups are finite.

There is a converse of sorts to part (2).
Let $\Gamma=F(2)/[F(2)'',F(2)]$.
Then $\zeta\Gamma=\Gamma''$ and is free abelian of countable rank.
Every countable abelian group $A\cong\Gamma''/R$ for some  $R\leq\Gamma''$.
Since $R$ is central and $\Gamma/\Gamma''=F(2)/F(2)''$ has trivial centre, 
it follows that $\zeta(\Gamma/R)\cong{A}$.
In fact, this is also $K_{\Gamma/R}$.
In particular, every countable torsion-free abelian group is $K_G$ for some 2-generator 
group with no non-trivial locally finite normal subgroup.
What countable abelian groups are centres of finitely presentable groups?

In part (4), $G/K_G$ has no non-trivial elementary amenable normal subgroup,
and so embeds in $Comm(G/K_G)$.
Is  $\rho$ an isomorphism?
In one significant case this is so.

\begin{theorem}
If there is a homomorphism $f:G\to{P}$ to a $PD_2$-group $P$ with finite kernel 
and $[P:f(G)]$ finite then $G$ is commensurable with $P$.
\end{theorem}

\begin{proof}
Let $H=C_G(K_G)$. Then $[G:H]$ is finite and $K_H=H\cap{K_G}=\zeta{H}$ 
is a finite abelian group, of exponent $e$, say.
On passing to a subgroup of finite index in $H$, 
we may assume that $\overline{H}=H/K_H$ is an orientable $PD_2$-group.
Let $\widetilde{H}$ be a subgroup of index $e$ in $\overline{H}$.
Since $\overline{H}$ and $\widetilde{H}$ are orientable surface groups,
restriction from $H^2(\overline{H};K_H)\cong{K_H}$ to 
$H^2(\widetilde{H};K_H)\cong{K_H}$ is multiplication by  the degree $e$.
Hence the characteristic class of the central extension
$0\to{K_H}\to{H}\to\overline{H}\to1$ restricts to 0.
Thus $H$ is commensurable with $\widetilde{H}\times{K_H}$,
and hence with $\widetilde{H}$ and then $P$.
\end{proof}

In general, commensurability is  a stricter equivalence relation than the one
determined by homomorphisms with finite kernel and image of finite index.

\begin{theorem}
Let $G$ be a group with a presentation of deficiency $\geq1$.
Then either
\begin{enumerate}
\item  $K_G$ is the maximal finite normal subgroup of $G$; or

\item $\mathrm{def(G)}=1$, $c.d.G\leq2$ and $K_G=\sqrt{G}\cong\mathbb{Z}$ 
or $\mathbb{Z}^2$.
\end{enumerate}
In the latter case either $G\cong\mathbb{Z}^2$ or $\mathbb{Z}\rtimes_{-1}\mathbb{Z}$
or $G$ is commensurable with $\mathbb{Z}\times{F(2)}$.
\end{theorem}

\begin{proof} If $K_G$ is not the maximal finite normal subgroup of $G$ then $G$
has an infinite amenable normal subgroup, by Theorem 11.
Hence $\beta^{(2)}_1(G)=0$ \cite{CG}, 
and so $\mathrm{def(G)}=1$ and $c.d.G\leq2$,
by Corollary 2.4.1 of \cite{Hi}.
In this case $K_G$ is nontrivial and torsion free.
Let $k$ be a nontrivial element of $K_G$.
Then $k$ centralizes some subgroup $H$ of finite index in $G$.
We may assume that $H$ is normal in $G$.
Then $k^{[G:H]}$ is a nontrivial element of $\zeta{H}$, 
so $\sqrt{G}\not=1$.
Either $\sqrt{G}\cong\mathbb{Z}$ or $G'$ is abelian, by Theorem 2.7 of \cite{Hi}.

If $\sqrt{G}\cong\mathbb{Z}$ then $[G:C_G(\sqrt{G}]\leq2$
and $G/\sqrt{G}$ is virtually free, 
by Theorem 8.4 of \cite{Bi}.
(In particular, the preimage in $G$ of a free subgroup $F$ of finite index in 
$C_G(\sqrt{G}/\sqrt{G}$ is isomorphic to $\mathbb{Z}\times{F}$.)
Hence $\sqrt{G}\leq{K_G}$.
Since $K_G/\sqrt{G}$ is a torsion group,
and is a normal subgroup of $G/\sqrt{G}$ it is finite.
Since $K_G$ is torsion free it follows that $K_G\cong\mathbb{Z}$ and hence that
$K_G=\sqrt{G}$.

If $G'$ is abelian and $\sqrt{G}$ is not infinite cyclic
then $G\cong\mathbb{Z}*_m$, the ascending HNN
extension with presentation $\langle a,t\mid tat^{-1}=a^m\rangle$, 
for some $m\not=0$, by Corollary 2.6 of \cite{Hi}.
But it is easy to see that $K_G$ must then be trivial unless $m=\pm1$,
in which case $G\cong\mathbb{Z}^2$ or $\mathbb{Z}\rtimes_{-1}\mathbb{Z}$ 
and $K_G=\sqrt{G}\cong\mathbb{Z}^2$.
\end{proof}

If $\mathrm{def(G)}\geq1$ and $K<G$ is a finite normal subgroup must it be trivial?
This would be generally true if it is so when $K$ is central.
For if $z\in{K}$ then $\widetilde{G}=C_G(z)$ also has deficiency $\geq1$,
and the subgroup $\widetilde{K}=\langle{z}\rangle$ is central.
If $G$ has a minimal generating set $\{x_1,\dots,x_r\}$ and $Z$ has order $p>1$
then one might expect that the relations $z^p=1$ and $zx_i=x_iz$ for $1\leq{i}\leq{r}$ 
are independent, and so $\mathrm{def(\widetilde{G})}\leq0$.
Although it is not clear how to substantiate this expectation in general,
an argument similar to that of Theorem 12  may be used to confirm it
when $G/K$ is commensurate with a $PD_2$-group \cite{Hi02}.
The analogue for pro-$p$-groups was proven in \cite{HS08}.

\end{document}